\documentclass[11pt,a4paper]{article}
\usepackage[utf8]{inputenc}
\usepackage{amsmath}
\usepackage{amssymb}
\usepackage{amsfonts}
\usepackage{amsthm}
\usepackage[english]{babel}
\usepackage{hyperref}
\usepackage{geometry}
\usepackage{graphicx}
\usepackage{tikz} 
\usepackage{booktabs} 
\usepackage{appendix} 

\newtheorem{theorem}{Theorem}
\newtheorem{lemma}{Lemma}

\theoremstyle{remark} 
\newtheorem*{remark}{Remark}
\newtheorem*{interpretation}{Geometric Interpretation}

\title{On a BBP-type formula for $\pi^2$ in the golden ratio base}
\author{B. Cloitre}
\date{July 31, 2025}

\begin{document}

\maketitle

\begin{abstract}
This paper presents a detailed, self-contained proof of a BBP-type formula for $\pi^2$ expressed in the golden ratio base, $\phi$. The formula was discovered empirically by the author in 2004. The proof presented herein is built upon a fundamental geometric identity connecting $\phi$ to the fifth roots of unity, offering an intuitive and direct path to the result. The power of the underlying methodology is then demonstrated by extending it to establish a new, computationally efficient Machin-like formula for $\zeta(3)$, expressed through rapidly converging, hierarchical series involving the golden ratio.
\end{abstract}

\section{Introduction and Context}

The search for explicit formulas for mathematical constants has long been a source of fascination. A significant breakthrough in this area came with the discovery of "BBP-type" formulas, named after Bailey, Borwein, and Plouffe \cite{BBP97}. A BBP formula for a constant $C$ is a series representation of the form:
\[ C = \sum_{k=0}^{\infty} \frac{1}{b^k} \sum_{j=1}^{n} \frac{a_j}{(nk+j)^s} \]
where $b$ is an integer base and $a_j$ are integers. Their primary importance lies in their ability to allow the computation of the $d$-th digit of $C$ in base $b$ without computing the preceding digits. The archetypal example is the formula for $\pi$:
\[ \pi = \sum_{k=0}^{\infty} \frac{1}{16^k} \left( \frac{4}{8k+1} - \frac{2}{8k+4} - \frac{1}{8k+5} - \frac{1}{8k+6} \right) \]
The original discovery for $\pi$ sparked a flurry of research. This led to many new BBP-type formulas, first for other weight-1 constants, and later for constants of higher weight, such as $\zeta(3)$ or Catalan's constant, as extensively documented in Bailey's compendium \cite{BaileyCompendium}. A parallel line of inquiry explored the existence of such formulas in non-integer algebraic bases. The golden ratio base, $\phi=(1+\sqrt{5})/2$, proved to be a particularly fertile ground, yielding formulas for $\pi$ and other weight-1 constants \cite{Chan08}.

This paper addresses a question at the confluence of these two research streams: do BBP-type formulas exist for constants of higher weight in the golden ratio base? It is in this context that in December 2004, while experimenting with the \texttt{lindep} algorithm in PARI/GP, I discovered the novel BBP-type formula for $\pi^2$—a constant of weight 2—that is the subject of this paper. While, as noted, numerous formulas for weight-1 constants have been found in the golden ratio base, those for constants of higher weight remain exceptionally rare. The formula for $\pi^2$ was subsequently featured as an exercise in Experimental Mathematics in Action  \cite{Bailey07}, where the hints provided are sufficient to construct a proof by specializing a general polylogarithmic identity.

In contrast, the demonstration provided herein offers a more fundamental and intuitive perspective. It originates from a single geometric identity (Lemma \ref{lem:geom_id}), which, through the properties of the dilogarithm, establishes a direct connection between a remarkable feature of the regular pentagon and the formula for $\pi^2$. The entire proof is thus constructed from first principles, highlighting the underlying algebraic structure in a self-contained way.

Finally, the investigation was extended to $\zeta(3)$, a constant of weight 3. While the experimental search for a pure BBP-type formula for $\zeta(3)$ in the golden ratio base did not succeed, it led to the discovery of a different, noteworthy result: a non-linear Machin-like formula for $\zeta(3)$. This identity, which expresses $\zeta(3)$ as a polynomial of rapidly converging series in terms of $\phi$, is established here as a valuable and computationally efficient result in its own right.

\section{Mathematical Preliminaries}

To make this paper self-contained, this section gathers the key definitions and identities used.

\subsection{The Polylogarithm Function}
The polylogarithm function, whose study was initiated by Euler and later developed by Landen and Lewin, provides a bridge between series, integrals, and special values of functions like the Riemann zeta function. For instance, $\mathrm{Li}_2(1) = \sum_{k=1}^\infty 1/k^2 = \zeta(2) = \pi^2/6$. These functions also appear in various contexts in theoretical physics, such as in the calculation of Feynman diagrams.

The polylogarithm of order $s$, denoted $\mathrm{Li}_s(z)$, is defined for any complex number $z$ with $|z|<1$ by the power series:
\[
\mathrm{Li}_s(z) = \sum_{k=1}^{\infty} \frac{z^k}{k^s}
\]
For $s=2$, this is the \textit{dilogarithm}, and for $s=3$, the \textit{trilogarithm}. 

\begin{lemma}[Functional Equations for the Dilogarithm]
For $|z| \le 1$ and $z$ not $0$ or $1$, the dilogarithm satisfies the following functional equations, which are standard results found in \cite{Lewin81}:
\begin{enumerate}
    \item \textbf{Landen's Identity (Reflection Formula):}
    \begin{equation}
    \mathrm{Li}_2(z) + \mathrm{Li}_2(1-z) = \frac{\pi^2}{6} - \ln(z)\ln(1-z)
    \label{eq:landen}
    \end{equation}
    \item \textbf{Inversion Formula:}
    \begin{equation}
    \mathrm{Li}_2(z) + \mathrm{Li}_2(1/z) = -\frac{\pi^2}{6} - \frac{1}{2}\ln^2(-z)
    \label{eq:inversion}
    \end{equation}
\end{enumerate}
Here, $\ln(z)$ denotes the principal branch of the complex logarithm.
\end{lemma}

\subsection{The Golden Ratio and the Regular Pentagon}
The golden ratio, $\phi = (1+\sqrt{5})/2$, is deeply connected to the geometry of the regular pentagon. The vertices of a regular pentagon centered at the origin can be represented by the fifth roots of unity, $\omega^k = e^{2\pi i k/5}$ for $k=0,1,2,3,4$. The geometry of this figure dictates the exact trigonometric values needed for our proof. For instance, the ratio of a diagonal to a side in a regular pentagon is $\phi$. This property can be used to derive the following standard results \cite{MathWorldPentagon}:
\[
\cos(2\pi/5) = \frac{\phi-1}{2} = \frac{1}{2\phi} \quad \text{and} \quad \cos(\pi/5) = \frac{\phi}{2}
\]
The point $z = \phi^{-1}e^{2\pi i/5}$, which is central to our proof, lies within this geometric context, as shown in Figure \ref{fig:pentagon}.

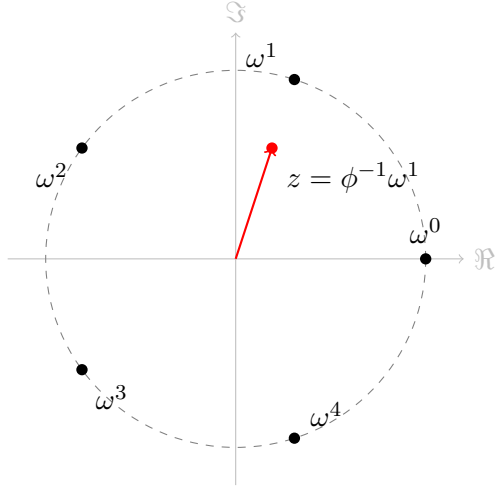
\begin{figure}[h!]
\centering
\begin{tikzpicture}[scale=2.5]
    \draw[->, gray!50] (-1.2,0) -- (1.2,0) node[right] {$\Re$};
    \draw[->, gray!50] (0,-1.2) -- (0,1.2) node[above] {$\Im$};
    \draw[dashed, gray] (0,0) circle (1);
    \foreach \k in {0,...,4} {
        \node[fill, circle, inner sep=1.5pt, label={90+72*\k:$\omega^{\k}$}] at ({72*\k}:1) {};
    }
    \def\phiinv{0.618034}
    \node[fill=red, circle, inner sep=1.5pt, label={-45:$z=\phi^{-1}\omega^1$}] at (72:\phiinv) {};
    \draw[->, red, thick] (0,0) -- (72:\phiinv);
\end{tikzpicture}
\caption{The fifth roots of unity $\omega^k=e^{2\pi i k/5}$ on the unit circle and the key point $z$ in the complex plane.}
\label{fig:pentagon}
\end{figure}

\section{The main results}

This paper establishes two main theorems.
\begin{theorem}[BBP formula for $\pi^2$]
Let $\phi = \frac{1+\sqrt{5}}{2}$ be the golden ratio. The following identity holds:
\begin{equation}
\frac{\pi^2}{50} = \sum_{k=0}^{\infty} \frac{1}{\phi^{5k}} \left( \frac{\phi^{-2}}{(5k+1)^2} - \frac{\phi^{-1}}{(5k+2)^2} - \frac{\phi^{-2}}{(5k+3)^2} + \frac{\phi^{-5}}{(5k+4)^2} + \frac{2\phi^{-5}}{(5k+5)^2} \right)
\label{eq:bbp_formula_main}
\end{equation}
\end{theorem}

\begin{theorem}[Machin-like formula for $\zeta(3)$]
Let $\phi = \frac{1+\sqrt{5}}{2}$. The constant $\zeta(3)$ can be expressed as the following polynomial of rapidly converging series:
\[
\zeta(3) = \frac{5}{4}M_3(\phi^{-2}) + \Big( M_2(\phi^{-1}) + M_2(\phi^{-2}) \Big) M_1(\phi^{-2}) + \frac{7}{6} \Big( M_1(\phi^{-2}) \Big)^3
\]
where the constants $M_s(w)$ are defined by hierarchical sums designed for rapid convergence:
\[
M_s(w) = \sum_{j=0}^{\infty} 5^{-sj} \left( \sum_{r=1}^{4}\sum_{k=0}^{\infty} \frac{(w^{5^j})^{5k+r}}{(5k+r)^s} \right)
\]
This series converges for $|w| < 1$, which is the case for the arguments $\phi^{-1}$ and $\phi^{-2}$ used here.
\end{theorem}

\section{Proof of the $\pi^2$ formula (theorem 1)}

The proof is presented in two stages: we first prove a core geometric identity and its consequence for a polylogarithm series, and then we demonstrate that the BBP formula is an algebraic rearrangement of this result.

\subsection{The Geometric and Polylogarithmic Foundation}

\begin{lemma}[A key geometric identity]
Let $\phi = (1+\sqrt{5})/2$ and $z = \phi^{-1}e^{2\pi i/5}$. Then, the following identity holds: $1-z = e^{-i\pi/5}$.
\label{lem:geom_id}
\end{lemma}

\begin{proof}
We prove the identity by showing that the real and imaginary parts of both sides are equal.
First, the real part of $1-z$:
\begin{align*}
\Re(1-z) &= 1 - \phi^{-1}\cos(2\pi/5) = 1 - \phi^{-1}\left(\frac{1}{2\phi}\right) = 1 - \frac{\phi^{-2}}{2} \\
&= 1 - \frac{1-\phi^{-1}}{2} = \frac{2 - 1 + \phi^{-1}}{2} = \frac{1+\phi^{-1}}{2} = \frac{\phi}{2}
\end{align*}
This is precisely $\cos(\pi/5) = \Re(e^{-i\pi/5})$.
Next, the imaginary part of $1-z$:
\[ \Im(1-z) = -\phi^{-1}\sin(2\pi/5) \]
Using the double-angle formula $\sin(2\theta) = 2\sin\theta\cos\theta$:
\begin{align*}
\phi^{-1}\sin(2\pi/5) &= \phi^{-1} \left(2\sin(\pi/5)\cos(\pi/5)\right) \\
&= \phi^{-1} \left(2\sin(\pi/5)\frac{\phi}{2}\right) = (\phi^{-1}\phi)\sin(\pi/5) = \sin(\pi/5)
\end{align*}
Thus, $\Im(1-z) = -\sin(\pi/5)$. Since the real and imaginary parts match, the identity is proven.
\end{proof}

\begin{interpretation}
This identity expresses a remarkable metric relationship within the geometry of the pentagon. In the complex plane, the numbers $0, 1, z$ form a triangle. The identity $1-z = e^{-i\pi/5}$ states that the vector from $z$ to $1$ has length $1$ and makes an angle of $-\pi/5$ with the real axis. This specific configuration is a direct consequence of the golden ratio's role in the pentagon's construction.
\begin{figure}[h!]
\centering
\begin{tikzpicture}[scale=3]
    \draw[->, gray!50] (-0.2,0) -- (1.2,0) node[right] {$\Re$};
    \draw[->, gray!50] (0,-0.6) -- (0,0.6) node[above] {$\Im$};
    \node[fill, circle, inner sep=1pt] at (0,0) (O) {}; \node[below left] at (O) {$0$};
    \node[fill, circle, inner sep=1pt] at (1,0) (one) {}; \node[below] at (one) {$1$};
    \def\phiinv{0.618034}
    \coordinate (z) at (72:\phiinv);
    \node[fill=red, circle, inner sep=1pt] at (z) {}; \node[above] at (z) {$z$};
    \draw[->, thick] (O) -- (one);
    \draw[->, red, thick] (O) -- (z);
    \draw[->, blue, thick] (z) -- (one) node[midway, below] {$1-z$};
    \coordinate (res) at (-36:1);
    \draw[->, blue, thick, dashed] (O) -- (res) node[midway, below right] {$e^{-i\pi/5}$};
\end{tikzpicture}
\caption{Geometric visualization of the identity $1-z = e^{-i\pi/5}$.}
\label{fig:geom_id}
\end{figure}
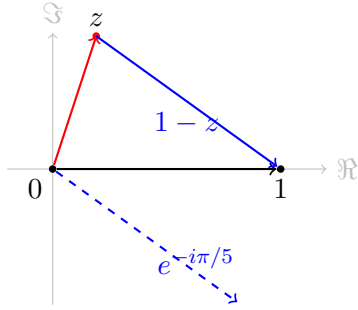
\end{interpretation}

\begin{lemma}
Let $\phi = \frac{1+\sqrt{5}}{2}$. The following identity holds:
\begin{equation}
\sum_{k=1}^{\infty} \frac{\cos(2\pi k/5)}{\phi^k k^2} = \frac{\pi^2}{100}
\label{eq:cos_sum}
\end{equation}
\end{lemma}

\begin{proof}
Our strategy is to express the target sum using dilogarithms and then use a functional equation to evaluate it. We choose the specific complex number $z = \phi^{-1}e^{2\pi i/5}$ because its real part isolates the desired cosine term, its modulus is $\phi^{-1}$, and, crucially, the term $1-z$ simplifies dramatically due to Lemma \ref{lem:geom_id}. This simplification makes Landen's identity, which connects $\mathrm{Li}_2(z)$ and $\mathrm{Li}_2(1-z)$, the ideal tool for the evaluation.

Let $z = \phi^{-1}e^{2\pi i/5}$ and its conjugate $\bar{z} = \phi^{-1}e^{-2\pi i/5}$. We want to evaluate the sum $S = \sum_{k=1}^{\infty} \frac{\cos(2\pi k/5)}{\phi^k k^2}$. This sum can be expressed using the dilogarithm as:
\[ 2S = \sum_{k=1}^{\infty} \frac{2\cos(2\pi k/5)}{\phi^k k^2} = \sum_{k=1}^{\infty} \frac{z^k + \bar{z}^k}{k^2} = \mathrm{Li}_2(z) + \mathrm{Li}_2(\bar{z}) \]
Our goal is to show that this sum equals $\pi^2/50$. We use Landen's Identity \eqref{eq:landen} and our key result from Lemma \ref{lem:geom_id}, $1-z = e^{-i\pi/5}$ (which implies $1-\bar{z} = e^{i\pi/5}$):
\begin{align*}
\mathrm{Li}_2(z) + \mathrm{Li}_2(\bar{z}) &= \left(\frac{\pi^2}{6} - \mathrm{Li}_2(1-z) - \ln(z)\ln(1-z)\right) + \left(\frac{\pi^2}{6} - \mathrm{Li}_2(1-\bar{z}) - \ln(\bar{z})\ln(1-\bar{z})\right) \\
2S &= \frac{\pi^2}{3} - \Big(\mathrm{Li}_2(e^{-i\pi/5}) + \mathrm{Li}_2(e^{i\pi/5})\Big) - \Big(\ln(z)\ln(1-z) + \ln(\bar{z})\ln(1-\bar{z})\Big)
\end{align*}
We evaluate the two parenthesized terms separately. First, the dilogarithm term is $2\Re(\mathrm{Li}_2(e^{i\pi/5}))$. The real part of $\mathrm{Li}_2(e^{i\theta})$ is the Clausen function, which has the known series representation $\sum \frac{\cos(k\theta)}{k^2} = \frac{\pi^2}{6} - \frac{\pi|\theta|}{2} + \frac{\theta^2}{4}$ for $0 \le |\theta| \le 2\pi$. For $\theta=\pi/5$, this term becomes:
\[ 2 \times \left(\frac{\pi^2}{6} - \frac{\pi(\pi/5)}{2} + \frac{(\pi/5)^2}{4}\right) = 2\pi^2\left(\frac{1}{6} - \frac{1}{10} + \frac{1}{100}\right) = 2\pi^2\left(\frac{50-30+3}{300}\right) = \frac{23\pi^2}{150} \]
Second, for the logarithm term, we use $\ln(z) = -\ln\phi + i\frac{2\pi}{5}$, $\ln(1-z) = -i\frac{\pi}{5}$, and their conjugates. The sum is:
\[ \left(-\ln\phi + i\frac{2\pi}{5}\right)\left(-i\frac{\pi}{5}\right) + \left(-\ln\phi - i\frac{2\pi}{5}\right)\left(i\frac{\pi}{5}\right) = \left(\frac{2\pi^2}{25} + i\frac{\pi\ln\phi}{5}\right) + \left(\frac{2\pi^2}{25} - i\frac{\pi\ln\phi}{5}\right) = \frac{4\pi^2}{25} \]
Substituting these results back into the expression for $2S$:
\[ 2S = \frac{\pi^2}{3} - \frac{23\pi^2}{150} - \frac{4\pi^2}{25} = \frac{50\pi^2 - 23\pi^2 - 24\pi^2}{150} = \frac{3\pi^2}{150} = \frac{\pi^2}{50} \]
Therefore, $S = \pi^2/100$, which completes the proof.
\end{proof}

\subsection{Derivation of the BBP formula}
Let $\mathcal{T} = \sum_{k=1}^{\infty} \frac{\cos(2\pi k/5)}{\phi^k k^2} = \frac{\pi^2}{100}$, as proven in Lemma 3. 
To transform this cosine series into a BBP-type formula, we use a standard dissection technique. We partition the sum based on the index $k$ modulo 5. This is the natural choice because the term $\cos(2\pi k/5)$ is periodic with period 5. This dissection isolates the powers of $\phi^5$, revealing the structure required for a BBP formula.

We set $k = 5j+r$ for $j \ge 0$ and $r \in \{1, 2, 3, 4, 5\}$:
\begin{align*}
\mathcal{T} &= \sum_{j=0}^{\infty} \sum_{r=1}^{5} \frac{\cos(2\pi(5j+r)/5)}{\phi^{5j+r} (5j+r)^2} 
= \sum_{j=0}^{\infty} \frac{1}{\phi^{5j}} \sum_{r=1}^{5} \frac{\cos(2\pi r/5)}{\phi^r (5j+r)^2}
\end{align*}
The values of $\cos(2\pi r/5)$ are given in Table \ref{tab:cos_vals}.

\begin{table}[h!]
\centering
\caption{Values of $\cos(2\pi r/5)$ used in the dissection.}
\label{tab:cos_vals}
\begin{tabular}{@{}ccc@{}}
\toprule
$r$ & $2\pi r/5$ & $\cos(2\pi r/5)$ \\
\midrule
1 & $2\pi/5$ & $(\phi-1)/2 = \phi^{-1}/2$ \\
2 & $4\pi/5$ & $-(\phi)/2$ \\
3 & $6\pi/5$ & $-(\phi)/2$ \\
4 & $8\pi/5$ & $(\phi-1)/2 = \phi^{-1}/2$ \\
5 & $2\pi$ & $1$ \\
\bottomrule
\end{tabular}
\end{table}

Substituting these values into the inner sum gives:
\begin{align*}
\sum_{r=1}^{5} \frac{\cos(2\pi r/5)}{\phi^r (5j+r)^2} &= \frac{\phi^{-1}/2}{\phi^1(5j+1)^2} + \frac{-\phi/2}{\phi^2(5j+2)^2} + \frac{-\phi/2}{\phi^3(5j+3)^2} + \frac{\phi^{-1}/2}{\phi^4(5j+4)^2} + \frac{1}{\phi^5(5j+5)^2} \\
&= \frac{1}{2} \left( \frac{\phi^{-2}}{(5j+1)^2} - \frac{\phi^{-1}}{(5j+2)^2} - \frac{\phi^{-2}}{(5j+3)^2} + \frac{\phi^{-5}}{(5j+4)^2} + \frac{2\phi^{-5}}{(5j+5)^2} \right)
\end{align*}
This is exactly half the term in the BBP formula \eqref{eq:bbp_formula_main}. Let $\mathcal{S}$ be the sum in that formula. Then we have shown that $\mathcal{T} = \mathcal{S}/2$. Since we proved $\mathcal{T} = \pi^2/100$, it follows that $\mathcal{S} = 2\mathcal{T} = \pi^2/50$. This completes the proof of Theorem 1.

\begin{remark}
The final term in the BBP polynomial, corresponding to $r=5$, can be written as a series involving the dilogarithm. The full sum associated with this term is $\sum_{j=0}^\infty \frac{2\phi^{-5}}{\phi^{5j}(5j+5)^2} = \frac{2\phi^{-5}}{25} \sum_{j=0}^\infty \frac{(\phi^{-5})^j}{(j+1)^2} = \frac{2}{25}\mathrm{Li}_2(\phi^{-5})$. This highlights that every component of the formula is structurally a polylogarithm value.
\end{remark}

\section{Proof of the $\zeta(3)$ formula (theorem 2)}

\begin{lemma}
For any complex number $w$ such that $|w|<1$ and any integer $s \ge 1$, we have $\mathrm{Li}_s(w) = M_s(w)$.
\label{lem:equiv_M_Li}
\end{lemma}
\begin{proof}
We begin with the definition of the polylogarithm, $\mathrm{Li}_s(w) = \sum_{n=1}^{\infty} \frac{w^n}{n^s}$. Any positive integer $n$ can be uniquely written as $n = m \cdot 5^j$, where $j \ge 0$ is an integer and $m$ is an integer not divisible by 5. We can thus re-index the sum by splitting it over $j$ and $m$:
\[ \mathrm{Li}_s(w) = \sum_{j=0}^{\infty} \sum_{\substack{m=1 \\ 5 \nmid m}}^{\infty} \frac{w^{m \cdot 5^j}}{(m \cdot 5^j)^s} = \sum_{j=0}^{\infty} \frac{1}{(5^s)^j} \sum_{\substack{m=1 \\ 5 \nmid m}}^{\infty} \frac{(w^{5^j})^m}{m^s} \]
The inner sum is over all integers $m$ not divisible by 5. This means $m$ can be written as $5k+r$ for $r \in \{1, 2, 3, 4\}$. The sum becomes:
\[ \sum_{\substack{m=1 \\ 5 \nmid m}}^{\infty} \frac{(w^{5^j})^m}{m^s} = \sum_{r=1}^4 \sum_{k=0}^\infty \frac{(w^{5^j})^{5k+r}}{(5k+r)^s} \]
Substituting this back into the expression for $\mathrm{Li}_s(w)$ gives exactly the definition of $M_s(w)$.
\end{proof}

\begin{lemma}[Key Polylogarithm Identities]
The proof relies on the following known evaluations of polylogarithm constants, found in sources such as \cite{Lewin81}.
\begin{enumerate}
    \item $\mathrm{Li}_3(\phi^{-2}) = \frac{4}{5}\zeta(3) + \frac{2}{3}(\ln\phi)^3 - \frac{2}{15}\pi^2\ln\phi$
    \item $\ln\phi = \mathrm{Li}_1(\phi^{-2})$
    \item $\pi^2 = 6\Big( \mathrm{Li}_2(\phi^{-1}) + \mathrm{Li}_2(\phi^{-2}) + 2(\ln\phi)^2 \Big)$
\end{enumerate}
Identity 3 is proven in Appendix A using standard dilogarithm evaluations (Eqs. \ref{eq:a1}-\ref{eq:a2}). For further details on all identities, see Appendix A.
\label{lem:polylog_ids}
\end{lemma}

\begin{proof}[Proof of theorem 2]
The proof follows a "bootstrapping" strategy. By Lemma \ref{lem:equiv_M_Li}, we know that $\mathrm{Li}_s(w) = M_s(w)$ for $|w|<1$, so we can use these expressions interchangeably in the following identities. We begin by rearranging the first identity in Lemma \ref{lem:polylog_ids} to isolate $\zeta(3)$:
\[ \zeta(3) = \frac{5}{4}\mathrm{Li}_3(\phi^{-2}) - \frac{5}{6}(\ln\phi)^3 + \frac{1}{6}\pi^2\ln\phi \]
Now, we substitute the expressions for $\ln\phi$ (identity 2) and $\pi^2$ (identity 3) into this equation.
\[ \zeta(3) = \frac{5}{4}\mathrm{Li}_3(\phi^{-2}) - \frac{5}{6}\big(\mathrm{Li}_1(\phi^{-2})\big)^3 + \frac{1}{6} \left[ 6\Big( \mathrm{Li}_2(\phi^{-1}) + \mathrm{Li}_2(\phi^{-2}) + 2(\mathrm{Li}_1(\phi^{-2}))^2 \Big) \right] \mathrm{Li}_1(\phi^{-2}) \]
Distributing the last term and combining the cubic terms yields the desired identity in terms of polylogarithms:
\[ \zeta(3) = \frac{5}{4}\mathrm{Li}_3(\phi^{-2}) + \Big(\mathrm{Li}_2(\phi^{-1}) + \mathrm{Li}_2(\phi^{-2})\Big)\mathrm{Li}_1(\phi^{-2}) + \frac{7}{6}\Big(\mathrm{Li}_1(\phi^{-2})\Big)^3 \]
To obtain the final formula stated in Theorem 2, we substitute each $\mathrm{Li}_s(w)$ term with its equivalent hierarchical series $M_s(w)$.
\end{proof}

\subsection*{Numerical illustration}
The efficiency of the formula for $\zeta(3)$ comes from the factor $5^{-sj}$ in the hierarchical sums $M_s(w)$. We know $\zeta(3) \approx 1.202056903159...$
\begin{itemize}
    \item The \textbf{j=0} term alone gives an approximation: $\zeta(3) \approx 1.202041...$
    \item Including the \textbf{j=1} term yields a much better result: $\zeta(3) \approx 1.202056902...$
    \item The \textbf{j=2} term adds a correction of order $10^{-33}$, yielding an accuracy of over 30 decimal places.\footnote{The high-order correction arises from the polynomial structure of the formula. While the correction from any single $M_s$ term is modest (e.g., the $j=2$ term in $M_3(\phi^{-2})$ is of order $10^{-15}$), the interplay between terms in the full polynomial for $\zeta(3)$ leads to significant cancellations, producing the much smaller final correction.}
\end{itemize}

\section{Conclusion and Perspectives}

We have presented a self-contained, geometrically-motivated proof of a BBP-type formula for $\pi^2$ in the golden ratio base. By building the proof from a specific identity tied to the regular pentagon, our approach offers an intuitive perspective on the formula's origin.

Furthermore, we have proved a new Machin-like identity for $\zeta(3)$. This work opens several avenues for reflection:
\begin{itemize}
    \item \textbf{Generalizations:} The success of this method is intimately tied to the specific geometric and algebraic properties of $\phi$ within the field of the fifth roots of unity. It seems unlikely that this approach could be generalized to other irrational algebraic numbers that do not possess such a remarkable geometric counterpart. The formula appears to be a "beautiful accident" of the pentagon's properties. One might wonder if analogous properties for other regular polygons (e.g., the octagon, related to $\sqrt{2}$) could yield similar formulas, but this remains a challenging open question.
    \item \textbf{Higher-order constants:} It is significant that our method yields a Machin-like formula for $\zeta(3)$ and not a pure BBP formula. While BBP formulas for $\zeta(3)$ in integer bases have been found \cite{BaileyCompendium}, the structure of polylogarithm identities related to $\phi$ seems to naturally produce the mixed terms characteristic of a Machin-like formula. This highlights a structural difference between representations in integer bases versus certain algebraic bases.
    \item \textbf{Computational aspects:} The formula for $\zeta(3)$ exhibits linear convergence. The number of correct digits obtained is roughly proportional to the number of terms computed in the outer sum ($j$). Specifically, each new term in $M_s(w)$ adds approximately $s\log_{10}(5)$ digits of precision. This is computationally effective for high-precision calculations, although its rate of linear convergence is lower than that of some known hypergeometric series for $\zeta(3)$. A detailed complexity analysis comparing this formula with other state-of-the-art algorithms/series would be a valuable next step.
\end{itemize}
On a pedagogical level, we hope this self-contained proof serves as an effective gateway for aspiring mathematicians, illustrating how a surprising numerical discovery can lead to the exploration of deep connections between analysis, number theory, and geometry.

\section*{Acknowledgements}
The author would like to thank Jonathan M. Borwein and Marc Chamberland for their interest in this formula. A special acknowledgement is dedicated to the memory of Jonathan M. Borwein, whose books and extensive online contributions have been a constant source of inspiration.

\begin{appendices}
\section{On the Polylogarithm Identities in Lemma \ref{lem:polylog_ids}}
To enhance the self-contained nature of this paper, we provide context for the three identities involving polylogarithm constants related to $\phi$. These are deep results from the theory of polylogarithms, with detailed proofs found in specialized literature such as Lewin's treatise \cite{Lewin81}.

\subsection{Identity 2: $\ln\phi = \mathrm{Li}_1(\phi^{-2})$}
This is the most straightforward identity. By definition, $\mathrm{Li}_1(z) = -\ln(1-z)$. We also have the algebraic property $\phi^{-2} = 1 - \phi^{-1}$.
Therefore:
\[ \mathrm{Li}_1(\phi^{-2}) = -\ln(1-\phi^{-2}) = -\ln(\phi^{-1}) = -(-\ln\phi) = \ln\phi \]

\subsection{Identity 3: $\pi^2 = 6\Big( \mathrm{Li}_2(\phi^{-1}) + \mathrm{Li}_2(\phi^{-2}) + 2(\ln\phi)^2 \Big)$}
This identity can be derived from two well-known evaluations for the golden ratio:
\begin{align}
\mathrm{Li}_2(\phi^{-2}) &= \frac{\pi^2}{15} - (\ln\phi)^2 \label{eq:a1} \\
\mathrm{Li}_2(\phi^{-1}) &= \frac{\pi^2}{10} - (\ln\phi)^2 \label{eq:a2}
\end{align}
Adding these two equations gives:
\[ \mathrm{Li}_2(\phi^{-1}) + \mathrm{Li}_2(\phi^{-2}) = \left(\frac{\pi^2}{15} + \frac{\pi^2}{10}\right) - 2(\ln\phi)^2 = \frac{2\pi^2+3\pi^2}{30} - 2(\ln\phi)^2 = \frac{5\pi^2}{30} - 2(\ln\phi)^2 = \frac{\pi^2}{6} - 2(\ln\phi)^2 \]
Rearranging this result yields the desired identity:
\[ \mathrm{Li}_2(\phi^{-1}) + \mathrm{Li}_2(\phi^{-2}) + 2(\ln\phi)^2 = \frac{\pi^2}{6} \implies \pi^2 = 6\Big( \mathrm{Li}_2(\phi^{-1}) + \mathrm{Li}_2(\phi^{-2}) + 2(\ln\phi)^2 \Big) \]
The proofs of \eqref{eq:a1} and \eqref{eq:a2} themselves rely on intricate manipulations of functional equations for the dilogarithm, such as the inversion formula and the five-term relation, evaluated at special points related to $\phi$.

\subsection{Identity 1: The $\mathrm{Li}_3(\phi^{-2})$ evaluation}
The evaluation of $\mathrm{Li}_3(\phi^{-2})$ in terms of $\zeta(3)$, $\pi^2$, and powers of $\ln\phi$ is a much deeper result. Its proof requires advanced techniques, including functional equations for the trilogarithm and connections to the theory of modular forms. A full derivation is well beyond the scope of this paper but can be found in advanced resources on polylogarithms \cite[Chapter 6]{Lewin81}.

\end{appendices}

\end{document}